\documentclass[12pt,a4paper]{amsart}
\usepackage{mathrsfs}
\usepackage{amssymb,amsmath,amsthm,color}
\usepackage{graphicx,mcite}
\usepackage{hyperref}
\usepackage{url}
\usepackage{setspace}
\newcommand{\loc}{\mbox{loc}}

\newtheorem{theorem}{Theorem}[section]
\newtheorem{lemma}[theorem]{Lemma}
\newtheorem{proposition}[theorem]{Proposition}
\newtheorem{corollary}[theorem]{Corollary}

\theoremstyle{definition}

\newtheorem{remark}[theorem]{Remark}

\numberwithin{equation}{section}

\begin{document}

\title[Injectivity of twisted spherical means]
{Non-harmonic cones are sets of injectivity for the twisted spherical means on $\mathbb C^n$}

\author{R. K. Srivastava}

\address{R. K. Srivastava, Department of Mathematics, Indian Institute of Technology, Guwahati, India 781039.}
\email{rksri@iitg.ernet.in}

\subjclass[2000]{Primary 43A85; Secondary 44A35}

\date{\today}


\keywords{Heisenberg group, spherical harmonics, twisted convolution.}

\begin{abstract}
In this article, we prove that a complex cone is a set of injectivity for the
twisted spherical means for the class of all continuous functions on $\mathbb C^n$
as long as it does not completely lay on the level surface of any bi-graded homogeneous
harmonic polynomial on $\mathbb C^n.$  Further, we produce examples of such level surfaces.
\end{abstract}

\maketitle

\section{Introduction}\label{section1}
Let $\mu_r$ be the normalized surface measure on the sphere $S_r(x).$ Suppose $F\subseteq L^1_{\tiny\loc}(\mathbb R^n).$
We say that $S\subseteq\mathbb R^n$ is a set of injectivity for the spherical means for $F$
if for $f\in F$ with $f\ast\mu_r(x)=0, \forall r>0$ and $\forall x\in S,$ implies $f=0.$

\smallskip
In an interesting article by Zalcman et al. \cite{AVZ}, it has been shown that
a real cone $C$ in $\mathbb R^d~(d\geq2)$ is a set of injectivity for the spherical
means for the class of all continuous functions on $\mathbb R^d$ if and only if $C$
does not lay on the level surface of any homogeneous harmonic polynomial on $\mathbb R^d.$
We look at this problem in the Heisenberg group set up. In particular, for the twisted
spherical mean on $\mathbb C^n.$ We prove an analogous result for the twisted spherical
means (TSM) for the class of all continuous functions on  $\mathbb C^n.$

\smallskip
Let $\mu_r$ be the normalized surface measure on the sphere $S_r(z)=\{w\in\mathbb C^n:~|z-w|=r\}.$
For a locally integrable function $f$ on $\mathbb C^n,$ we define its twisted
spherical mean on the sphere $S_r(z)$ by
\[f\times\mu_r(z)=\int_{|w|=r}~f(z-w)e^{\frac{i}{2}\text{Im}(z.\bar{w})}~d\mu_r(w).\]

\smallskip
Let $F\subseteq L^1_{\tiny\loc}(\mathbb C^n).$  We say $S\subseteq\mathbb C^n$
is a set of injectivity for twisted spherical means for $F$ if for
$f\in F$ with $f\times\mu_r(z)=0, \forall r>0$ and $\forall z\in S,$
implies $f=0$ a.e.  The results on set of injectivity differ in the choice
of sets and the class of functions considered. We would like to refer to
\cite{AR, NT1, Sri}, for some results on the sets of injectivity for the TSM.

\smallskip

A set $C$ in $\mathbb C^n~(n\geq 2)$ which satisfy $\lambda C\subseteq C,$
for all $\lambda\in\mathbb C$ is called a complex cone. In this article we
have proved the following result. Let $f$ be a continuous function on
$\mathbb C^n.$ Suppose $f\times\mu_r(z)=0,$ for all $r>0$ and $z\in C.$
Then $f\equiv0$ as as long as $C$ does not lay on the level surface of any bi-graded
homogeneous harmonic polynomial on $\mathbb C^n.$ We will call such cones as
\textbf{\em non-harmonic} cones. A proof of this result for functions in
$L^p(\mathbb C^n),~1\leq p\leq2,$ has been worked out using real analytic expansion
of the spectral projections, (see \cite{Sri3}). However, this proof does not stand
for the class of all continuous functions on $\mathbb C^n.$

\smallskip

We would like to assert that our main result of this article is in sharp contrast
(in terms of topological dimension) with the Euclidean results about the sets of injectivity
for the spherical means. Since a non-trivial complex cone in $\mathbb C^n~(n\geq2)$ can
have topological dimension at most $2n-2,$ therefore, it follows that a $(2n-2)-$ dimensional
entity is a set of injectivity for the TSM on $\mathbb C^n.$ All though, for the Euclidean
set up in $\mathbb R^{2d},$ the least topological dimension required (in general) for a set
to be set of injectivity for the spherical means is $2d-1.$ For instance, the boundary of a
bounded domain in $\mathbb R^{2d}$ is a set of injectivity for the spherical means for
$L^p(\mathbb R^{2d})$ with $1\leq p\leq\frac{4d}{2d-1}.$
(See \cite{ABK}). The main result of this article is also distinct to the known results
in terms of topological dimension of sets of injectivity for the TSM. The topological
dimension of sets of injectivity for the TSM on $\mathbb C^n$ for the known results is
$2n-1.$ For example, boundary of bounded domain and $(\mathbb R\cup i\mathbb R)\times\mathbb C^{n-1}$ are
sets of injectivity for the TSM in $\mathbb C^n$ having topological dimension $2n-1.$
 For more details see, \cite{AR, NT1, Sri}.

\smallskip

On the basis of our result that any non-harmonic complex cone is a set of injectivity
for the TSM, we can pose the following interesting question. A non-trivial complex
submanifold $S$ of $\mathbb C^n~(n\geq2)$ is a set injectivity for the TSM if and only if
$S$ does not contained in the zero set of any bi-graded homogeneous harmonic polynomial.
We leave this question open for the time being.

\smallskip

In order to complete the argument of this result, we show that the zero set of
polynomial $H(z)=az_1\bar z_2+|z|^2,$ where $a\not=0$ and $z\in\mathbb C^n$ is a
complex cone which does not contained in the zero set of any bi-graded homogeneous
harmonic polynomial. The proof of the fact that $H^{-1}(0)$ is a non-harmonic cone
is one of the most difficult pat of this article. We call a complex cone to be
{\em non-harmonic} if it does not contained in the zero set of any bi-graded
homogeneous harmonic polynomial on $\mathbb C^n$ of the form
\begin{equation}\label{exp11}
P(z)=\sum_{|\alpha|=p}\sum_{|\beta|=q}c_{\alpha\beta}z^\alpha\bar{z}^\beta,
\end{equation}
where $p,q\in\mathbb Z_+,$ the set of non-negative integers. Let $P_{p,q}$
denotes the space of all bi-graded homogeneous polynomials defined by (\ref{exp11}).
For $p\geq1,$ we observe that each of the diagonal space $P_{p,p}$ has at least one
member which corresponds to a non-harmonic complex cone. However, for $p, q\geq1,$
it is intersecting to know that whether the non-diagonal space $P_{p,q}$
has at least one member which corresponds to a non-harmonic complex cone.

\smallskip

The question of sets of injectivity for the spherical means has been taken up by many
authors in recent past. In an article by Agranovsky et al. \cite{ABK}, it has been proved
that the boundary of any bounded domain in $\mathbb R^d~(d\geq2)$ is set of injectivity for
the spherical means on $L^p(\mathbb R^d),$ with $1\leq p\leq\frac{2d}{d-1}.$

\smallskip

In general, characterization of sets of injectivity for spherical mean is a long
standing problem in integral geometry. In fact, an intuition about sets of injectivity
is that these sets are essentially those sets which are seating out side zero set of a
homogeneous harmonic polynomials union an algebraic variety of codimension at most $d-2.$
In 1996, Agranovsky and Quinto have given a characterization of sets of injectivity for
spherical mean for compactly supported function on $\mathbb R^2$ in terms of zero set of
a homogeneous harmonic polynomials union a finite set. They have conjectured this question
in higher dimension in similar way, (see \cite{AQ}). Following is their result.
\begin{theorem}\label{th7}\cite{AQ}
A set $S\subset\mathbb R^2$ is a set of injectivity for the spherical means for
$C_c(\mathbb R^2)$ if and only if $S\nsubseteq\omega(\Sigma_N)\cup F,$ where $\omega$
is a rigid motion of $\mathbb R^2, \Sigma_N=\cup_{l=0}^{N-1}\{te^{\frac{i\pi l}{N}}: t\in\mathbb R\}$
is a  Coxeter system of $N$ lines and $F$ is a finite set in $\mathbb R^2.$
\end{theorem}
Since, a homogeneous harmonic polynomial can be expressed as the product of homogeneous
polynomials of degree $1.$ Therefore, $\Sigma _N$ in Theorem \ref{th7} is nothing but
zero set of some homogeneous harmonic polynomial on $\mathbb R^2.$
Further studies on sets of injectivity for spherical mean and related problems have
been carried out in \cite{AK,AN,AQ,NRR,PS,V}.
Later, the question of sets of injectivity for the TSM has been considered by Agranovsky
and Rawat, (see \cite{AR}). They have shown that the boundary of any bounded domain in
$\mathbb C^n$ is set of injectvity for the TSM for a certain class of functions in
$L^p(\mathbb C^n).$ For more histories and further work on this question, we refer
\cite{AN1,AN2,NT1, Sri, Sri2, Sri3, ST, V1, V2, V3, V4,V5}.

\smallskip

A continuous function $f$ on $\mathbb R^d$  can be decomposed in terms of
spherical harmonics as
\begin{equation}\label{exp9}
f(x)=\sum_{k=0}^\infty a_{k,j}(\rho) Y_{k,j}(\omega),
\end{equation}
where $x=\rho\omega,~\rho=|x|,~\omega\in S^{n-1}$ and $\{Y_{kj}:
1,2,\ldots, d_k\}$ is an orthonormal basis for the space $V_k$ of the homogeneous
harmonic polynomials in $d$ variables of degree $k,$ restricted to the unit sphere
$S^{n-1}$ with the series in the right-hand side converges locally uniformly to $f.$
Let $H_k=\left\{P_k: P_k(x)=\rho^kY_k(\omega), Y_k\in V_k\right\}.$ The space $H_k$ is called
the space of solid spherical harmonics. For more details, see \cite{SW}.

\smallskip

Let $\nu_r$ be the normalized surface measure on the sphere $S_r(x)$ in $\mathbb R^d.$
For $f\in L^1_{\tiny\loc}(\mathbb R^d),$ we define its mean over the sphere $S_r(x)$ by
\[f\ast\nu_r(x)=\int_{S_r(x)}f(y)d\mu_r(y).\]

A set $C$ in $\mathbb R^d~(d\geq2)$ which satisfies $\lambda C\subseteq C,$
for all $\lambda\in\mathbb R$ is called a real cone.  Zalcman et al. \cite{AVZ}
have proved the following result.

\begin{theorem}\label{th3}\cite{AVZ}.
Let $C$ be a real cone in $\mathbb R^d~(d\geq2).$  Let $f$ be a continuous
function on $\mathbb R^d.$ Suppose $f\times\nu_r(x)=0,$ for all $r>0$ and
$x\in C.$ Then $f\equiv0$ if and only if $C\nsubseteq P^{-1}(0),$ for any
$P\in H_k$ and for all $k\in\mathbb Z_+.$
\end{theorem}
An example of such a cone had been produced by Armitage, (see \cite{A}). Let $0<a<1$
and $C_k^\lambda(x)$ denotes Gegenbauer polynomial of degree $k$ and order $\lambda.$
Then $K_a=\left\{x\in\mathbb R^d:~|x_1|^2=a^2|x|^2\right\}$ is a non-harmonic cone if and only
if $D^m C_k^{\frac{d-2}{2}}(a)\neq0,$ for all $0\leq m\leq k-2,$ where $D^m$ denotes
the $m$th derivative.

\smallskip
We would like to mention that the proof of Theorem \ref{th3} is being deduced
by concentrating the problem to the unit sphere $S^{d-1}$ in terms of average on
its geodesic spheres. This is possible because cone $C$ is closed under scaling.
This lemma is also useful in the proof of our main result.

\smallskip
For $\omega\in S^{d-1}$ and $t\in(-1, 1),$ set  $S_\omega^t=\left\{v\in S^{d-1}: \omega\cdot v=t\right\}$
is a geodesic sphere on $S^{d-1}$ with pole at $\omega.$ Let $f$ be a continuous function
on $S^{d-1}.$ Define
\[\tilde f(\omega, t)=\int_{S_\omega^t}f d\nu_{d-2},\]
where $\nu_{n-2}$ is the normalized surface measure on the geodesic sphere $S_\omega^t.$

\begin{lemma}\label{lemma3}\cite{AVZ}
Suppose $f\in C(S^{d-1})$ have spherical harmonic expansion $\sum_{k=0}^\infty Y_k.$
Then $\tilde f(\omega, t)=0,~\forall~t\in (-1, 1)$
if and only if $Y_k(\omega)=0,~\forall k\in\mathbb Z_+.$ In particular,
if $\tilde f(\omega, t)=0,~\forall~t\in (-1, 1)$ then $f\equiv0$ if and
only if $\omega$ is not contained in the zero set of any homogeneous
harmonic polynomial.
\end{lemma}
The results on the sphere $S^{d-1}$ for geodesic mean need not be same as results for spherical mean
on $\mathbb R^d.$ For instance, Theorem \ref{th7} says that sets of injectivity for spherical mean are essentially
(up to a translation and rotation) sitting out side the zero set of a homogeneous harmonic polynomials union a finite set.

\section{Preliminaries}\label{section2}

We define the twisted convolution which arises in the study of
group convolution on Heisenberg group. The group $\mathbb H^n,$
as a manifold, is $\mathbb C^n \times\mathbb R$ with the group law
\[(z, t)(w, s)=\left(z+w,t+s+\frac{1}{2}\text{Im}(z.\bar{w})\right),~z,w\in\mathbb C^n\text{ and }t,s\in\mathbb R.\]
Let $\mu_s$ be the normalized surface measure on the sphere $ \{ (z,0):~|z|=s\} \subset \mathbb H^n.$
The spherical means of a function $f$ in $L^1(\mathbb H^n)$ are defined by
\begin{equation} \label{exp22}
f\ast\mu_s(z, t)=\int_{|w|=s}~f((z,t)(-w,0))~d\mu_s(w).
\end{equation}
Thus the spherical means can be thought of as convolution operators. An important technique in many
problem on $\mathbb H^n$ is to take partial Fourier transform in the $t$-variable to reduce matters
to $\mathbb C^n$. Let
\[f^\lambda(z)=\int_\mathbb R f(z,t)e^{i \lambda t} dt\]
be the inverse Fourier transform of $f$ in the $t$-variable. Then a simple calculation shows that

\begin{eqnarray*}
(f \ast \mu_s)^\lambda(z)&=&\int_{-\infty}^{~\infty}~f \ast \mu_s(z,t)e^{i\lambda t} dt\\
&=&\int_{|w| = s}~f^\lambda (z-w)e^{\frac{i\lambda}{2} \text{Im}(z.\bar{w})}~d\mu_s(w)\\
&=&f^\lambda\times_\lambda\mu_s(z),
\end{eqnarray*}
where $\mu_s$ is now being thought of as normalized surface measure
on the sphere $S_s(o)=\{z\in\mathbb C^n: |z|=s\}$ in $\mathbb C^n.$
Thus the spherical mean $f\ast \mu_s$ on the Heisenberg group can be
studied using the $\lambda$-twisted spherical mean $f^\lambda
\times_\lambda\mu_s$ on $\mathbb C^n.$ For $\lambda \neq 0,$
by scaling argument, it is enough to study the
twisted convolution for the case $\lambda=1.$

\smallskip

We need the following basic facts from the theory of bigraded
spherical harmonics, (see \cite{D, Gr1, T} for details). Let $K=U(n)$
and $M=U(n-1).$ Then, $S^{2n-1}\cong K/M$ under
the map $kM\rightarrow k.e_n,$ $k\in U(n)$ and $e_n=(0,\ldots
,1)\in \mathbb C^n.$ Let $\hat{K}_M$ denote the set of all
equivalence classes of irreducible unitary representations of $K$
which have a nonzero $M$-fixed vector.

For a $\delta\in\hat{K}_M,$ which is realized on $V_{\delta},$ let
$\{e_1,\ldots, e_{d(\delta)}\}$ be an orthonormal basis of
$V_{\delta}$ with $e_1$ as the $M$-fixed vector. Let
$t_{ij}^{\delta}(k)=\langle e_i,\delta (k)e_j \rangle ,$ $k\in K.$
By Peter-Weyl theorem, it follows that $\{\sqrt{d(\delta
)}t_{j1}^{\delta}:1\leq j\leq d(\delta ),\delta\in\hat{K}_M\}$ form an
orthonormal basis for $L^2(K/M)$ (see \cite{T}, p.14 for details).
Define $Y_j^{\delta} (\omega )=\sqrt{d(\delta )}t_{j1}^{\delta}(k),$
where $\omega =k.e_n\in S^{2n-1},$ $k \in K.$ Then
$\{Y_j^{\delta}:1\leq j\leq d(\delta ),\delta\in \hat{K}_M, \}$
becomes an orthonormal basis for $L^2(S^{2n-1}).$

For our purpose, we need a concrete realization of the
representations in $\hat{K}_M,$ which can be done in the following
way. See \cite{Ru}, p.253, for details.
For $p,q\in\mathbb Z_+$, let $H_{p,q}=\{P\in P_{p,q}:\Delta P=0\}.$
The group  $K$ acts on $H_{p,q}$ in a natural way. It is easy to see that the space
$H_{p,q}$ is $K$-invariant. Let $\pi_{p,q}$ denote the corresponding
representation of $K$ on $H_{p,q}.$ Then representations in
$\hat{K}_M$ can be identified, up to unitary equivalence, with the
collection $\{\pi_{p,q}: p,q \in \mathbb Z_+\}.$

Define the bi-graded spherical harmonic by $Y_j^{p,q}(\omega
)=\sqrt{d(p,q )}t_{j1}^{p,q}(k).$
Then $\{Y_j^{p,q}:1\leq j\leq d(p,q),p,q \in \mathbb Z_+ \}$ forms
an orthonormal basis for $L^2(S^{2n-1}).$ Therefore, a continuous
function $f$ on $S^{2n-1}$  can be expressed in terms of bi-graded
spherical harmonics as
\begin{equation}\label{Bexp4}
f (\omega) = \sum_{p,q\geq0} \sum_{j=1}^{d(p,q)} Y_j^{p,q}(\omega).
\end{equation}

\bigskip

For each $k,$ the space $V_k$  is invariant under the action of $SO(d).$
When $d = 2n,$ it is invariant under the the action of the unitary
group $U(n)$ as well, and under this action of $U(n)$ the space $V_k$
breaks up into an orthogonal direct sum of $H_{p,q}$'s where $p+q = k.$
(See \cite{Ru}, p. 255).

\begin{lemma}\label{lemma4}\cite{Ru}.
Let $\omega\in S^{2n-1}$ and $Y_k\in V_k.$ Then
\[Y_k(\omega)=\sum_{p+q=k}Y_{p,q}(\omega), \text{ where }Y_{p,q}\in H_{p,q}.\]
\end{lemma}

As a consequence to Lemma \ref{lemma4}, we prove the following lemma
which is required in the proof of our main result.
\begin{lemma}\label{lemma5}
Let $\Omega=\left\{\frac{z}{|z|}:~z\in C, ~ z\neq0\right\}.$ Then
$Y_k=0$ on $\Omega$ if and only if $Y_{p,q}=0$ on $\Omega,~\forall~
p, q\in\mathbb Z_+$ such that $p+q=k.$
\end{lemma}

\begin{proof}
Let $\omega\in\Omega$ and $Y_k(\omega)=0.$ Then by Lemma \ref{lemma4}, we have
\[\sum_{p+q=k}Y_{p,q}(\omega)=0.\]
Since the cone $C$ is closed under complex scaling, by replacing
$\omega$ for $e^{i\theta}\omega$ in the above equation, we get
\[\sum_{p+q=k}e^{i(p-q)\theta}Y_{p,q}(\omega)=0.\]
Using the fact that the set $\{e^{i\beta\theta}: ~\beta\in\mathbb Z\}$
form an orthogonal set and the sum vanishes on the diagonal $p+q=k,$
we conclude that $Y_{p,q}(\omega)=0,~\forall~p, q\in\mathbb Z_+$
such that $p+q=k.$
\end{proof}

We shall frequently need the following lemma which gives an unique
decomposition of homogeneous polynomial in terms of homogeneous harmonic
polynomials.

\begin{lemma}\label{lemma2}
Let $P\in P_{p,q}.$ Then $P(z)=P_0(z)+|z|^2P_1(z)+\cdots+|z|^{2l}P_l(z),$ where
$P_j\in H_{p-j, q-j};~j=1,2,\ldots, l$ and $l\leq \min(p,q).$
\end{lemma}
For a proof of this lemma, see \cite{T}, p. 66.

\smallskip

We shall also require the next two lemmas in the proof of  existence of non-harmonic
complex cones in $\mathbb C^n.$ Using the fact that Laplacian is rotation invariant,
we prove the following lemma.
\begin{lemma}\label{lemma1}
Let $0\leq j\leq\min(p, q)$ and $R_j\in H_{p-j, q-j}.$ Then
\[\Delta\left(|z|^{2j}R_j\right)=c_j|z|^{2j-2}R_j,\text{ where } c_j=4j(n+p+q-j-1).\]
\end{lemma}

\begin{proof}
Consider
\begin{equation}\label{exp9A}
  \Delta\left(|z|^{2j}R_j\right)=4\sum_{k=1}^n\frac{\partial^2}{\partial z_k\partial\bar z_k}\left(|z|^{2j}R_j\right).
\end{equation}
We have
\begin{eqnarray*}
\frac{\partial^2}{\partial z_k\partial\bar z_k}\left(|z|^{2j}R_j\right)
&=&\frac{\partial}{\partial z_k}\left\{j|z|^{2(j-1)}z_kR_j+ |z|^{2j}\frac{\partial R_j}{\partial\bar z_k}\right\}\\
&=&j\left\{(j-1)|z|^{2(j-2)}|z_k|^2R_j+ |z|^{2(j-1)}R_j+|z|^{2(j-1)}z_k\frac{\partial R_j}{\partial z_k}\right\}\\
&+& j(|z|^2)^{j-1}\bar z_k\frac{\partial R_j}{\partial\bar z_k}+|z|^{2j}\frac{\partial^2R_j}{\partial z_k\partial\bar z_k}.\\
\end{eqnarray*}
By using Euler's formula for homogeneous function, we get
\[\sum z_k\frac{\partial R_j}{\partial z_k}=(p-j)R_j \text{ and }
\sum\bar z_k\frac{\partial R_j}{\partial\bar z_k}=(q-j)R_j.\]
From Equation (\ref{exp9A}), it follows that
\[\Delta\left(|z|^{2j}R_j\right)=4j(n+p+q-j-1)|z|^{2j-2}R_j.\]
\end{proof}

For multi-indexes $\alpha, \beta\in\mathbb Z_+^n,$ write $\partial^\alpha=\partial_1^{\alpha_1}\cdots\partial_n^{\alpha_n}$
and $\bar\partial^\beta=\bar\partial_1^{\beta_1}\cdots\bar\partial_n^{\beta_n}.$ Then for $R\in P_{p,q},$
we can write
\[R(\partial)=\sum_{|\alpha|=p}\sum_{|\beta|=q}c_{\alpha\beta}\partial^\alpha\bar\partial^\beta.\]
For $R, S\in P_{p,q},$ define an inner product on  $P_{p,q}$ by
$\langle  R, S\rangle= R(\partial)\bar S.$ We shall prove the following lemma which is crucial
for proof of the fact that complex cone $H^{-1}(0)$ is a non-harmonic cone in $\mathbb C^n.$
\begin{lemma}\label{lemma6}
For $z_1, z_2\in\mathbb C,$ define $\mathcal A=\bar z_2\dfrac{\partial}{\partial \bar z_1}+z_1\dfrac{\partial}{\partial z_2}.$
Then $\mathcal A$ is a self-adjoint operator on the space $P_{p,q}.$
\end{lemma}
\begin{proof}
Consider
\[\left\langle  z_1\frac{\partial R}{\partial z_2}, S\right\rangle
=\left(\frac{\partial R}{\partial z_2}\right)(\partial)\frac{\partial}{\partial z_1}\bar S
=\left(\frac{\partial R}{\partial z_2}\right)(\partial) \frac{\overline{\partial S}}{\partial \bar z_1}.\]
That is,
\[\left\langle  z_1\frac{\partial R}{\partial z_2}, S\right\rangle
=\left\langle \frac{\partial R}{\partial z_2}, \frac{\partial S}{\partial \bar z_1}\right\rangle=
\overline{\left(\frac{\partial S}{\partial \bar z_1}\right)(\partial) \overline{\frac{\partial R}{\partial z_2}}}.\]
This implies,
\[\left\langle  z_1\frac{\partial R}{\partial z_2}, S\right\rangle=\overline{\left(\frac{\partial (\bar z_2S)}{\partial \bar z_1}\right)(\partial) \bar R}
=\left\langle R, \frac{\partial(\bar z_2S)}{\partial \bar z_1}\right\rangle.\]
Hence,
\[\left\langle  z_1\frac{\partial R}{\partial z_2}, S\right\rangle=\left\langle R, \bar z_2\frac{\partial S}{\partial \bar z_1}\right\rangle.\]

Similarly, we can obtain the equation
\[\left\langle  \bar z_2\frac{\partial R}{\partial\bar z_1}, S\right\rangle
=\left\langle R, z_1\frac{\partial S}{\partial z_2}\right\rangle.\]
Thus by combining both these conditions,  we get $\langle\mathcal A R, S\rangle=\langle R, \mathcal A S\rangle.$
That is, operator $\mathcal A$ is self-adjoint.
\end{proof}

\bigskip

We also need an expansion of functions on $\mathbb C^n$ in terms of Laguerre
functions $\varphi_k^{n-1}$'s, which is know as special Hermite expansion.
The special Hermite expansion is a useful tool in the study of convolution
operators and is related to the spectral theory of sub-Laplacian on the
Heisenberg group $H^n$. However, more details can be found in \cite{T}.

\smallskip

For $\lambda\in\mathbb R^*=\mathbb R\setminus\{0\}$, let
$\pi_\lambda$ be the unitary representation of $H^n$ on $L^2(\mathbb
R^n)$ given by
$$\pi_\lambda(z,t)\varphi(\xi)=e^{i\lambda t}e^{i\lambda(x.\xi+\frac{1}{2}x.y)}\varphi(\xi+y),
\varphi\in L^2(\mathbb R^n).$$ A celebrated theorem of Stone and von Neumann
says that $\pi_\lambda$ is  irreducible and up to unitary equivalence
$\{\pi_\lambda:~\lambda\in\mathbb R\}$ are all the infinite dimensional unitary
irreducible representations of $H^n.$
Let \[T=\dfrac{\partial}{\partial t},X_j=\dfrac{\partial}{\partial
x_j}+\dfrac{1}{2}y_j\dfrac{\partial}{\partial t}~,
Y_j=\dfrac{\partial}{\partial
y_j}-\dfrac{1}{2}x_j\dfrac{\partial}{\partial t},~j=1,2,\ldots,n.\]
Then  $\{T,X_j,Y_j:~j=1,\ldots,n\}$ is a  basis for the  Lie Algebra
$\mathfrak h^n$ of all left invariant vector fields on $H^n.$
Define $\mathcal L=-\sum_{j=1}^n(X_j^2+Y_j^2),$ the second order
differential operator which is known as the sub-Laplacian of $H^n.$
The representation $\pi_\lambda$ induces a representation
$\pi_\lambda^*$ of $\mathfrak h^n,$ on the space of $C^\infty$
vectors in $L^2(\mathbb R^n)$ is defined by
\[\pi_\lambda^*(X)f=\left.\frac{d}{dt}\right\vert_{t=0}\pi_\lambda(\exp tX)f.\]
An easy calculation shows that
$\pi^*(X_j)=i\lambda x_j,~\pi^*(Y_j)=\dfrac{\partial}{\partial x_j},~j=1,2,\ldots,n$.
Therefore, $\pi_\lambda^*(\mathcal
L)=-\Delta_x+\lambda^2|x|^2=:H(\lambda),$ the scaled Hermite
operator. The eigenfunction of $H(\lambda)$ are given by
$\phi_\alpha^\lambda(x)=|\lambda|^{\frac{n}{4}}\phi_\alpha(\sqrt{|\lambda|} x), ~\alpha\in\mathbb Z_+^n,$
where $\phi_\alpha$ are the Hermite functions on $\mathbb R^n.$
Since
$H(\lambda)\phi_\alpha^\lambda=(2|\lambda|+n)|\lambda|\phi_\alpha^\lambda.$
Therefore,
\[\mathcal L\left(\pi_\lambda(z,t)\phi_\alpha^\lambda,\phi_\beta^\lambda\right)
=(2|\lambda|+n)|\lambda|\left(\pi_\lambda(z,t)\phi_\alpha^\lambda,\phi_\beta^\lambda\right).\]
Thus the entry functions
$\left(\pi_\lambda(z,t)\phi_\alpha^\lambda,\phi_\beta^\lambda\right),~
 \alpha,\beta\in\mathbb Z_+^n$ are eigenfunctions for $\mathcal L$.
As
$\left(\pi_\lambda(z,t)\phi_\alpha^\lambda,\phi_\beta^\lambda\right)
 =e^{i\lambda t}\left(\pi_\lambda(z)\phi_\alpha^\lambda,\phi_\beta^\lambda\right),$
these eigenfunctions are not in $L^2(H^n).$ However for a fix $t,$
they are in $L^2(\mathbb C^n).$ Define  $\mathcal L_\lambda$ by
$\mathcal L\left(e^{i\lambda t}f(z)\right)=e^{i\lambda t}L_\lambda
f(z).$ Then the functions
\[\phi_{\alpha\beta}^\lambda(z)=
(2\pi)^{-\frac{n}{2}}\left(\pi_\lambda(z)\phi_\alpha^\lambda,\phi_\beta^\lambda\right),\]
are eigenfunction of the operator $L_\lambda$ with eigenvalue $2|\lambda|+n.$
The functions $\phi^\lambda_{\alpha\beta}$'s are called the special Hermite
functions and they form an orthonormal basis
for $L^2(\mathbb C^n)$ (see \cite{T}, Theorem 2.3.1, p.54). Thus,
for $g\in L^2(\mathbb C^n)$, we have the expansion
\[g=\sum_{\alpha,\beta}\left\langle g,\phi^\lambda_{\alpha\beta}\right\rangle\phi^\lambda_{\alpha\beta}.\]
To further simplify this expansion, let
$\varphi^{n-1}_{k,\lambda}(z)=\varphi^{n-1}_k(\sqrt{|{\lambda|}}z),$
the Leguerre function of degree $k$ and order $n-1$. The special
Hermite functions $\phi^\lambda_{\alpha\alpha}$ satisfy the relation
\begin{equation}\label{ACexp4}
\sum_{|\alpha|=k}\phi^\lambda_{\alpha,\alpha}(z)=(2\pi)^{-\frac{n}{2}}
|\lambda|^{\frac{n}{2}}\varphi^{n-1}_{k,\lambda}(z).
\end{equation}
Let $g$ be a function in $L^2(\mathbb C^n)$. Then $g$ can be expressed as
\[g(z)=(2\pi)^{-n}|\lambda|^n\sum_{k=0}^\infty g\times_\lambda\varphi_{k,\lambda}^{n-1}(z),\]
whenever $\lambda\in\mathbb R^*,$ (see \cite{T}, p.58). In particular, for $\lambda=1$, we have
\begin{equation}\label{ACexp16}
g(z)=(2\pi)^{-n}\sum_{k=0}^\infty g\times\varphi_k^{n-1}(z),
\end{equation}
which is called the special Hermite expansion for $g$.
For radial functions, this expansion further simplifies as can be seen from the following lemma.
\begin{lemma} \label{lemma1C}\emph{\cite{T}}
Let  $f$ be a radial function in $L^2(\mathbb C^n)$. Then
\[f=\sum_{k=0}^\infty B^n_k\left\langle f,\varphi^{n-1}_k\right\rangle\varphi^{n-1}_k,
~\text{where}~B^n_k=\frac{k!(n-1)!}{(n+k-1)!}.\]
\end{lemma}

\section{Proofs of the main results}\label{section3}
In this section, we first prove our main result that a
non-harmonic complex cone is a set of injectivity for the twisted
spherical means for the class of all continuous functions on $\mathbb C^n.$
After that we prove the existence of examples of non-harmonic complex cone.
In fact, we show that each of the diagonal space $P_{p,p}$ has at least one
member which corresponds to a non-harmonic complex cone. At the end we mention
some remarks and open problems related to the problem of sets of injectivity for the TSM.

\begin{theorem}\label{th1}
Let $C$ be a complex cone in $\mathbb C^n~(n\geq 2).$  Let $f$ be a continuous
function on $\mathbb C^n.$ Suppose $f\times\mu_r(z)=0,$ for all $r>0$ and
$z\in C.$ Then $f\equiv0$ if and only if $C\nsubseteq H^{-1}(0),$ for any
$P\in H_{p,q}$ and for all $p, q\in\mathbb Z_+.$
\end{theorem}

\begin{proof}
Since $C$ is closed under complex scaling, by rotation we can assume that
$z=(z_1,0,\ldots,0)\in C,$ for all $z_1\in\mathbb C.$ By the hypothesis
$f\times\mu_s(z)=0,~\forall s>0$ and  for all $z\in C.$ Therefore, we can write
\begin{equation}\label{exp1c}
\int_{|w|\leq r}f(z+w)e^{-\frac{i}{2}\text{Im}(z.\bar{w})}dw=
\int_{0}^rf\times\mu_s(z)s^{2n-1}dw=0,
\end{equation}
for all $r>0$ and $z\in C.$ Let $z_1=x_1+iy_1.$ Applying
$2\partial_{z_1}=2\frac{\partial}{\partial z_1}=\frac{\partial}{\partial x_1}-i\frac{\partial}{\partial y_1}$
to the above equation, we get
\[\int_{|w|\leq r}\frac{\partial}{\partial w_1}\left(f(z+w)e^{-\frac{i}{2}\text{Im}(z.\bar{w})}\right)dw
-\frac{1}{2}\int_{|w|\leq r}\bar w_1f(z+w)e^{-\frac{i}{2}\text{Im}(z.\bar{w})}dw=0,\]
for all $r>0$ and $z$ in $C.$ By Green's theorem, we get
\[\int_{|w|=r}\frac{\bar w_1}{r}\left(f(z+w)e^{-\frac{i}{2}\text{Im}(z.\bar{w})}\right)dw
=\frac{1}{2}\int_{|w|\leq r}\bar w_1f(z+w)e^{-\frac{i}{2}\text{Im}(z.\bar{w})}dw,\]
Let $g(z)=\bar z_1f(z).$ Then we have
\[r^{2n-2}g\times\mu_r(z)=\frac{1}{2}\int_{0}^rg\times\mu_s(z)s^{2n-1}ds.\]
Put $F(t)=t^{2n-1}g\times\mu_t(z).$ Then the above equation becomes
\[\frac{F(r)}{r}=\frac{1}{2}\int_{0}^rF(s)ds.\]
By differentiating both sides, we get
\[F'(r)=\left(\frac{r}{2}+\frac{1}{r}\right)F(r).\]
A general solution to this equation is
\[F(r)=\frac{c(z)}{r}e^{\frac{r^2}{4}}.\]
That is,
\[r^{2n-2}g\times\mu_r(z)=c(z)e^{\frac{r^2}{4}}.\]
By letting $r\rightarrow0,$ we get $c(z)=0.$ Hence  $g\times\mu_r(z)=0,$ for all $r>0$ and
$z\in C.$ Let $h(z)=z_1f(z).$ Similarly, by applying
$2\partial_{\bar z_1}=2\frac{\partial}{\partial\bar z_1}=\frac{\partial}{\partial x_1}+i\frac{\partial}{\partial y_1}$
to the equation (\ref{exp1c}), we get $h\times\mu_r(z)=0,$ for all $r>0$ and
$z\in C.$ Hence for any polynomial $P(z_1,\bar z_1),$ we deduce that
$(Pf)\times\mu_r(z)=0,~\forall ~r>0$ and $z\in C.$
Let $w_j=u_j+iv_j, ~j=1,2, \ldots,n.$ Since $0\in C,$ by evaluating
the means  at $0,$ we get
\[\int_{rS^{2n-1}}P(u_1)f(w)d\mu_r(w)=0,\] for all $r>0.$  Thus, we can approximate the
above equation at $u_1=t.$ For $\epsilon>0,$ we can write
\[\int_{rS^{2n-1}}\frac{1}{\epsilon~\sqrt\pi}~e^{-\frac{(u_1-t)^2}{\epsilon^2}}f(w)d\mu_r(w)=0.\]
Letting $\epsilon\rightarrow0,$ we get
\[\int_{rS^{2n-1}\cup\{u_1=t\}}fd\tilde\mu=0,\]
for all $t\in (-r, r),$ where $\tilde\mu$ is the normalized surface measure on the
geodesic sphere $rS^{2n-1}\cup\{u_1=t\}.$ Thus, the integral of $f$ vanishes over
all $(2n-2)$- dimensional geodesic spheres on $rS^{2n-1}$ with poles lay on
$rS^{2n-1}.$ In view of Lemma \ref{lemma3}, we infer that $f$ vanishes
on $rS^{2n-1}$ if and only if poles are not contained in $Y_k^{-1}(0),~\forall~k\in\mathbb Z_+.$
Further by Lemma \ref{lemma4}, we have \[Y_k=\sum_{p+q=k}Y_{p,q}.\] Hence, in view of Lemma \ref{lemma5},
we get $f=0$ on $rS^{2n-1}$ if and only if poles are not contained in
$Y_{p,q}^{-1}(0),~\forall~p, q\in\mathbb Z_+.$ Since $r>0$ is arbitrary and $f$ is continuous,
we conclude that $f\equiv0$ on $\mathbb C^n.$ This completes the proof.
\end{proof}

\begin{remark}
In Theorem \ref{th1}, we have shown that complex cone is set of injectivity
for the twisted spherical means, however the question of the real cones to be
sets of injectivity for the twisted spherical means is still unsolved.
For instance, the author in the article \cite{Sri2} has shown that the real
cone $\mathbb R\cup i\mathbb R\times\mathbb C^{n-1}$
is a set of injectivity for the TSM for the class $L^p(\mathbb C^n)~(n\geq2)$
with $1\leq p\leq2.$  This real cone is not contained in the zero set of
any bi-graded homogeneous harmonic polynomial. We would like to mention that
the later result is a consequence of a result that $\mathbb R\cup i\mathbb R$
is set of injectivity for the TSM for $L^p(\mathbb C),$ which has been proved
by the author in the article \cite{Sri}.
\end{remark}

In order to complete the argument of Theorem \ref{th1}, we now prove that
there exists a non-trivial complex cone which does not vanish on the zero set (or level surface)
of any bi-graed homogeneous harmonic polynomial.

Let $0\neq a\in\mathbb C$ and $z\in\mathbb C^n,~(n\geq2).$ Write
$H(z)=az_1\bar z_2+|z|^2.$ Then $H^{-1}(0)$ is a complex cone. In fact, we shall show
that this is a non-harmonic complex cone. Since $H$ is a homogeneous polynomial, for $n=2,$
put $z_1=wz_2.$ Then $H(wz_2, z_2)=|z_2|^2(w\bar w+aw+1).$ Since the polynomial
$w\bar w+aw+1$ can not be factorized into linear factors, therefore, $H$ is irreducible.
It is clear from the context that $H$ is also irreducible for $n>2.$ Now, it only
remains to show that $H^{-1}(0)\nsubseteq R^{-1}(0),$ for any $R\in H_{s,t}$ and for
all $s,t\in\mathbb Z_+.$ Otherwise, if $H^{-1}(0)\subseteq R^{-1}(0),$ then by using
the fact that $H$ is irreducible, it  implies that $H$ divides $R$ and hence $R=HQ,$
for some $Q\in P_{p,q}.$ Thus, it is enough to prove the following result in order
to prove our claim.

\begin{theorem}\label{th2}
Suppose $\Delta(HQ)=0,$ for some $Q\in P_{p,q}$ with $p,q\in\mathbb Z_+.$
Then $Q$ has to vanish identically.
\end{theorem}
In order to prove Theorem \ref{th2}, we need the following lemma. Write
\[\mathcal A=\bar z_2\frac{\partial}{\partial \bar z_1}+z_1\frac{\partial}{\partial z_2}
\text{ and } {\mathcal B}=\frac{\partial^2}{\partial \bar z_1\partial z_2}.\]

\begin{lemma}\label{lemma8}
Let $\Delta(HQ)=0,$ for some $Q\in P_{p,q}.$ Then
\begin{equation}\label{exp8}
HQ=a\left(z_1\bar z_2Q_0+ \gamma|z|^2\mathcal AQ_0 +\delta|z|^4{\mathcal B} Q_0\right),
\end{equation}
where $\gamma$ and $\delta$ are non-zero constants and independent of $a.$
\end{lemma}

\begin{proof}
Since $Q\in P_{p,q},$ therefore by Lemma \ref{lemma2}, $Q$ can be uniquely decomposed in terms of
spherical harmonics as
\[Q=Q_0+|z|^2Q_1+|z|^4Q_2+\cdots+|z|^{2l}Q_l,\]
where $Q_j\in H_{p-j, q-j}.$
We can write
\begin{eqnarray*}
HQ &=& az_1\bar z_2\left(Q_0+|z|^2Q_1+|z|^4Q_2+\cdots+|z|^{2l}Q_l\right)+ |z|^2Q \\
&=& az_1\bar z_2Q_0+ |z|^2\left(az_1\bar z_2Q_1+|z|^2az_1\bar z_2Q_2+\cdots+|z|^{2l-2}az_1\bar z_2Q_l+ |z|^2Q\right).
\end{eqnarray*}
Using Lemma \ref{lemma2} once again to the above equation, we get
\begin{equation}\label{exp5}
HQ=az_1\bar z_2Q_0+ |z|^2R_1+|z|^4R_2+\cdots+|z|^{2m}R_m,
\end{equation}
where $R_j\in H_{p-j, q-j}.$  Now, using the condition that $HQ$ is harmonic, we
show that $HQ$ is completely determined by $Q_0,$ which is the key part of this proof.
In fact, the later argument would enable us to assume that $Q$ is harmonic.

Since $\Delta(HQ)=0.$ Therefore, by Lemma \ref{lemma1}, it follows that
\begin{equation}\label{exp1}
\Delta (HQ)=a\Delta\left(z_1\bar z_2Q_0\right)+ c_1R_1+c_2|z|^2R_2+\cdots+c_m|z|^{2m-2}R_m=0.
\end{equation}
We have
\begin{eqnarray}\label{exp10}
\Delta\left(z_1\bar z_2Q_0\right)
&=&4\left[\frac{\partial^2(z_1\bar z_2Q_0)}{\partial z_1\partial\bar z_1} +
\frac{\partial^2(z_1\bar z_2Q_0)}{\partial z_2\partial\bar z_2}+
\sum_{k=3}^n\frac{\partial^2(z_1\bar z_2Q_0)}{\partial z_k\partial\bar z_k}\right] \nonumber\\
&=&4\bar z_2\frac{\partial Q_0}{\partial \bar z_1}+4z_1\frac{\partial Q_0}{\partial z_2}=4\mathcal AQ_0. \nonumber\\
\end{eqnarray}
By equation (\ref{exp1}), we have
\[a\mathcal AQ_0+ c_1R_1+c_2|z|^2R_2+\cdots+c_m|z|^{2m-2}R_m=0.\]
By applying $\Delta$ in the above equation, we get
\begin{equation}\label{exp3}
a\Delta\left(\mathcal AQ_0\right)+c_2c_2'R_2+c_3c_3'|z|^2R_3\cdots+c_mc_m'|z|^{2m-4}R_m=0.
\end{equation}
A straightforward calculation gives
\begin{equation}\label{exp2}
\Delta\left(\mathcal AQ_0\right)=4\mathcal A\Delta Q_0+8{\mathcal B} Q_0= 8{\mathcal B} Q_0,
\end{equation}
Operator $\mathcal B$ seems to be a degree reducing operator over the spaces $P_{p,q}$'s.
By combining (\ref{exp3}) and (\ref{exp2}) we get
\begin{equation}\label{exp4}
8a{\mathcal B} Q_0+c_2c_2'R_2+c_3c_3'|z|^2R_3\cdots+c_mc_m'|z|^{2m-4}R_m=0.
\end{equation}
Since the spaces $H_{p,q}$'s are orthogonal among themselves, it follows that
\[8a{\mathcal B} Q_0+c_2c_2'R_2=0,~R_3=0,\ldots, R_m=0.\]
By substituting these values in (\ref{exp5}), we get
\begin{equation}\label{exp7}
HQ=az_1\bar z_2Q_0+ |z|^2R_1-\frac{8a}{c_2c_2'}|z|^4{\mathcal B} Q_0.
\end{equation}
Once again applying $\Delta$ to (\ref{exp7}), we get
\[0=4a\mathcal AQ_0+4(p+q)R_1-\frac{32a}{c_2c_2'}(p+q+1)|z|^2{\mathcal B} Q_0.\]
Finally, after substituting the value of $R_1$ to (\ref{exp7}), we can write
\begin{equation}\label{exp8A}
HQ=a\left(z_1\bar z_2Q_0+ \gamma|z|^2\mathcal AQ_0 +\delta|z|^4{\mathcal B} Q_0\right),
\end{equation}
where $\gamma$ and $\delta$ are non-zero constants and independent of $a.$
\end{proof}

\noindent{{\em Proof of Theorem 3.3}}.
From Lemma \ref{lemma8}, we infer that  $HQ$ is completely determined by $Q_0.$ From (\ref{exp8}) and
in view of Lemma \ref{lemma2}, it is clear that $HQ=$ harmonic part of
$(az_1\bar z_2Q_0),$ because the rest of the terms will not contribute to $HQ,$
since $HQ$ is harmonic. This says that product $HQ$ depends only upon harmonic
part $Q_0$ of $Q$ and hence, without loss of generality, we can assume $Q$  to be
a homogeneous harmonic polynomial. That is, $\Delta Q=0.$

\smallskip

By the given condition, we have $\Delta(HQ)=0.$ Therefore, by using the condition that
$Q$ is harmonic, we get
\[ a\Delta\left(z_1\bar z_2Q\right)+\Delta\left(|z|^2 Q\right)=0.\]
In view of equation (\ref{exp10}) and Lemma \ref{lemma1}, we obtain that
\begin{equation}\label{exp6}
 a\mathcal AQ+(n+p+q)Q=0.\\
\end{equation}
This says that $Q$ is an eigenfunction of the operator $\mathcal A.$
Since by Lemma \ref{lemma6}, $\mathcal A$ is a self-adjoint operator, therefore,
its eigenvalues must be real.
From (\ref{exp6}), we have
\[\mathcal AQ= -\dfrac{(n+p+q)}{a}Q.\]
If $a=\alpha+i\beta$ and $\beta\neq0.$ Then $Q$ has to be identically zero. If $\beta=0,$
then we can use the rotation
\[\sigma=\left(
  \begin{array}{ccccc}
    e^{\frac{i\theta}{2}} & 0 & 0\\
    0 & e^{\frac{i\theta}{2}} &0 \\
    0 & 0 & I_{n-2} \\
    \end{array}
\right) \in U(n),\]
where $\theta$ is any real number other than even multiple of $\pi.$
This gives $$\sigma\cdot\left(z_1, z_2, z_3,\ldots,z_n\right)=
\left({e^{\frac{i\theta}{2}}z_1, e^{-\frac{i\theta}{2}}z_2,z_3, \ldots, z_n}\right).$$
Since $\Delta (HQ)=0$ and $\sigma\in U(n),$ therefore,  by using the fact that $\Delta$
is rotation invariant, it implies that $\sigma\cdot\Delta(HQ)=0.$ That is,
$\Delta(HQ)(\sigma\cdot z)=0,$ which in turn implies that
\[\Delta\left\{\left(ae^{i\theta}z_1\bar z_2+|z|^2\right)Q^\sigma(z)\right\}=0,\]
where $Q^\sigma(z)=Q(\sigma\cdot z).$
By a similar calculation as to the previous case, we can write
\[(n+p+q)Q^\sigma+ae^{i\theta}\mathcal A Q^\sigma=0.\]
Thus $Q^\sigma$ is an eigenfunction of $\mathcal A$ with complex eigenvalue having
non-zero imaginary part. Therefore,  $Q^\sigma=0$  and hence $Q=0.$ Thus, we infer that
Theorem \ref{th2} is valid for any complex number $a\neq0.$

\begin{corollary}\label{cor1}
Let $s\in\mathbb Z_+$ and $s\geq1.$ Then there exists $P_o\in P_{s, s}$ such that
$P_o^{-1}(0)$ is a non-harmonic complex cone.
\end{corollary}

\begin{proof}
As in Theorem \ref{th2}, let $H(z)=az_1\bar z_2+|z|^2.$ We show that
$H^s$ is a required member which belongs to $P_{s,s}$ such that $\left(H^s\right)^{-1}(0)$
is a non-harmonic complex cone. We prove this using induction on $s.$ For $s=1,$
by Theorem \ref{th2}, $H\in P_{1,1}$ and $H^{-1}(0)$ is non-harmonic complex cone.
When $s=2,$ we need to show that $\left(H^2\right)^{-1}(0)\nsubseteq R^{-1}(0),$ for any
$R\in H_{s,t}$ and for all $s,t\in\mathbb Z_+.$ On the contrary, suppose
$\left(H^2\right)^{-1}(0)\subseteq R^{-1}(0),$ for some $R\in H_{s,t}.$ Then
either $H$ divide $R$ or $H^2$ divide $R.$ Then as a consequence of Theorem \ref{th2},
it implies that $R=0.$ That is,  $H$ does not divide $R.$ Therefore,
 $H^2$ divide $R$ and hence $R=H^2Q$ for some $Q\in P_{p,q}.$ This in turn
implies that $\Delta\{H(HQ)\}=0.$ By applying Theorem \ref{th2} once again, we get $HQ=0.$ Hence $Q=0.$
This proves the result for $s=2.$ Similarly, the result for an arbitrary $s$ is being followed
by induction on $s.$
\end{proof}

\begin{remark}
$(a).$ In Corollary \ref{cor1}, we have shown that for $s\geq1,$ each of the diagonal
space $P_{s,s}$ has at least one member which corresponds to a non-harmonic complex cone.
Hence there are plenty of complex cones which do not completely lay in to the level surface
of any bi-graded homogeneous harmonic polynomials. Roughly speaking, there are plenty of
homogeneous surfaces other than that  homogeneous harmonic surfaces. However, it is
intersecting to know that whether a given non-diagonal space $P_{s,t}$ has at least
one member which corresponds to a non-harmonic complex cone. This question is open
for the time being.

\smallskip
$(b).$ In a recent article \cite{Sri2}, we have observed some embedding property of
set of injectivity in higher dimensions. For instance, the set $\mathbb R\cup i\mathbb R$
is a set of injectivity for the TSM for $L^p(\mathbb C)$ with $1\leq p\leq2.$ Then we deduced
that the set $(\mathbb R\cup i\mathbb R)\times\mathbb C^n$ is a set of injectivity for the TSM
for $L^p(\mathbb C^{n+1})$ with $1\leq p\leq2.$ Therefore, for a given non-harmonic cone $C$
in $\mathbb C^n,$ it is interesting to know that whether the set $C\times\mathbb C^m$ is a set
of injectivity for the TSM for the certain class of continuous functions on $\mathbb C^{n+m}.$
This question is still unanswered and for the time being we leave it open.
\end{remark}

\section{Weighted twisted spherical mean}\label{section4}
In this section, we prove that complex cone is a set of injectivity for certain
weighted twisted spherical mean for the radial class of functions which satisfy
some exponential growth condition.

\smallskip

Let $P\in H_{p,q}$ and denote $d\nu_r=Pd\mu_r.$ Let $C$ be a complex cone in $\mathbb C^n.$
Then we have the following result.

\begin{theorem}\label{th4}
Let $f$ be a radial function on $\mathbb C^n$ such that $ e^{\frac{1}{4}|z|^2}f(z)\in L^p(\mathbb C^n),$
for $1\leq p<\infty.$ Suppose $f\times\nu_r(z)=0, ~\forall ~r>0$ and $\forall ~z\in C.$
Then $f=0$ a.e. if and only if $C\nsubseteq P^{-1}(0).$
\end{theorem}
Theorem \ref{th4}, does not hold for $p=\infty$ as can be seen from the following weighted
functional equations for the spherical function $\varphi_k^{n-1}.$
\begin{lemma}\label{lemma3C}\emph{\cite{T}}
For $z\in\mathbb C^n,$ let $P\in H_{p,q}$ and $d\nu_r=Pd\mu_r$.
Then
\[\varphi_k^{n-1}\times\nu_r(z)=
(2\pi)^{-n}C(n,p,q)r^{2(p+q)}\varphi_{k-q}^{n+p+q-1}(r)P(z)\varphi_{k-q}^{n+p+q-1}(z),\]
if $k\geq q$ and ~$0$ otherwise.
\end{lemma}
\begin{remark}\label{Crk1}
From Lemma \ref{lemma3C}, it can  be seen that Theorem \ref{th4} does not hold for
$p=\infty.$ For instance, take $P\in H_{0,1}$ and let
$d\nu_r=Pd\mu_r$. Then $\varphi_0^{n-1}\times\nu_r(z)=0$, where
$\varphi_0^{n-1}(z)=e^{-\frac{1}{4}|z|^2}.$
\end{remark}

For the proof of Theorem \ref{th4}, we need the following result from \cite{Sri3} which is
related to the Weyl correspondence of bi-graded spherical harmonics. Let us consider the
following invariant differential operators which arises in study of the twisted convolution
on $\mathbb C^n.$ For $\lambda\in\mathbb R\setminus\{0\},$ let
\[\widetilde Z_{j,\lambda}=\frac{\partial}{\partial z_j}-\frac{\lambda}{4}\bar z_j \text{ and }
\widetilde Z_{j,\lambda}^{\ast}=\frac{\partial}{\partial\bar z_j}+\frac{\lambda}{4} z_j, ~j=1,2,\ldots, n.\]
Let $P$ be a bi-graed homogeneous harmonic polynomial on $\mathbb C^n$ with expression
\[P(z)=\sum_{|\alpha|=p}\sum_{|\beta|=q}c_{\alpha\beta}z^\alpha\bar{z}^\beta.\]
Then by using a result of Geller (\cite{Ge}, p.616, Proposition 2.7.) about Weyl correspondence
of the spherical harmonics, the operator analogue of $P(z)$  can be expressed as
\[P(\widetilde Z)=\sum_{|\alpha|=p}\sum_{|\beta|=q}c_{\alpha\beta}{\widetilde {Z^{\ast}}}^\alpha\widetilde Z^\beta.\]
The following result has been proved by the author in article \cite{Sri3}.
\begin{theorem}\cite{Sri3}\label{th5}
Let $P\in H_{p,q}$. Then
\begin{equation}\label{exp17}
P(\widetilde Z_\lambda)\varphi_{k,\lambda}^{n-1}=\left\{
  \begin{array}{ll}
   (-2\lambda)^{-p-q}P\varphi_{k-p,\lambda}^{n+p+q-1} , & \text{if } \lambda<0, ~k\geq q ; \\
   (-2\lambda)^{-p-q}P\varphi_{k-q,\lambda}^{n+p+q-1} , & \text{if } \lambda>0,  ~k\geq p.
  \end{array}
 \right.
\end{equation}
\end{theorem}
Consider the following right invariant differential operators for the twisted convolution:
\[\tilde A_j=\frac{\partial}{\partial z_j}+\frac{1}{4}\bar{z}_j ~\mbox{and}~\tilde A_j^*=
\frac{\partial}{\partial\bar{z}_j}-\frac{1}{4}z_j;~j=1,2,\ldots,n.\]
Let $\varphi_k^{n-1}(z)=L_k^{n-1}(\frac{1}{2}|z|^2)e^{-\frac{1}{4}|z|^2}.$
Then by Theorem \ref{th5}, we get
\begin{equation}\label{exp18}
P(\widetilde A)\varphi_k^{n-1}=(2)^{-p-q}P\varphi_{k-q}^{n+p+q-1} , \text{if }~k\geq q.
\end{equation}

Suppose $f$ be a function on $\mathbb C^n$ such that $e^{\frac{1}{4}|z|^2}f(z)\in L^p(\mathbb C^n),$
for $1\leq p<\infty.$ Let $\varphi_\epsilon$ be a smooth, radial compactly supported
approximate identity on $\mathbb C^n.$ Then $f\times\varphi_\epsilon\in L^1\cap L^\infty(\mathbb C^n)$
and in particular $f\times\varphi_\epsilon\in L^2(\mathbb C^n).$ Let $d\nu_r=Pd\mu_r.$ Suppose
$f\times\nu_r(z)=0, \forall r>0$ and $\forall z\in C.$ Then by polar decomposition
$f\times P\varphi_{k-q}^{n+p+q-1}(z)=0, \forall k\geq q$ and $\forall z\in C.$
Since $\varphi_\epsilon$ is radial, we can write
\[f\times\varphi_\epsilon\times\nu_r(z)=\sum_{k\geq 0}B_k^n\left\langle\varphi_\epsilon,\varphi_k^{n-1}\right\rangle
f\times\varphi_k^{n-1}\times\nu_r(z).\]
By Lemma \ref{lemma3C}, it follows that
$f\times\varphi_\epsilon\times\nu_r(z)=0, \forall k\geq q$ and $\forall z\in C.$
Thus, without loss of generality, we can assume $f\in L^2(\mathbb C^n).$
Hence to prove the Theorem \ref{th4}, it is enough to prove the
following result.

\begin{proposition}\label{prop1}
Let $f\in L^2(\mathbb C^n)$ be radial and $ e^{\frac{1}{4}|z|^2}f(z)\in L^p(\mathbb C^n),$
for $1\leq p<\infty.$ Suppose $f\times\nu_r(z)=0, ~\forall ~r>0$ and $\forall ~z\in C.$
Then $f=0$ a.e. if and only if $C\nsubseteq P^{-1}(0).$
\end{proposition}

\begin{proof}
Given that $f\times\nu_r(z)=0, ~\forall ~r>0$ and $\forall ~z\in C.$ By polar decomposition
$f\times P\varphi_{k-q}^{n+p+q-1}(z)=0, \forall k\geq q$ and $\forall z\in C.$
Using the identity (\ref{exp18}), we can write
$f\times P(\widetilde A)\varphi_{k}^{n-1}(z)=0, \forall k\geq q$ and $\forall z\in C.$
Since $ P(\widetilde A)$ is a right invariant operator, it follows that
\begin{equation}\label{exp19}
P(\widetilde A)\left(f\times\varphi_{k}^{n-1}\right)(z)=0.
\end{equation}
 Since $f$ is radial, therefore, it can be expressed as
\[f=\sum_{k=0}^\infty B^n_k\left\langle f,\varphi^{n-1}_k\right\rangle\varphi^{n-1}_k.\]
Since the Laguerre functions satisfy the orthogonality conditions
$\varphi_{k}^{n-1}\times\varphi_{j}^{n-1}=(2\pi)^{n}\delta_{jk}\varphi_{k}^{n-1}.$
Therefore, $f\times\varphi_{k}^{n-1}=(2\pi)^{n}B^n_k\left\langle f,\varphi^{n-1}_k\right\rangle\varphi^{n-1}_k.$
Hence, from (\ref{exp19}) it implies that
$\left\langle f,\varphi^{n-1}_k\right\rangle P(\widetilde A)\varphi^{n-1}_k(z)=0.$
Once again by applying the identity (\ref{exp18}), we get
\[\left\langle f,\varphi^{n-1}_k\right\rangle P(z)\varphi^{n+p+q-1}_{k-q}(z)=0.\]
That is,
\[\left\langle f,\varphi^{n-1}_k\right\rangle P(z)L^{n+p+q-1}_{k-q}\left(\frac{1}{2}|z|^2\right)=0,\]
whenever, $k\geq q$ and $ z\in C.$ Since $C$ is closed under scaling, it follows that
\[\left\langle f,\varphi^{n-1}_k\right\rangle P(z)L^{n+p+q-1}_{k-q}\left(\frac{1}{2}r^2|z|^2\right)=0,\]
for all $r>0.$ Thus, we get $\left\langle f,\varphi^{n-1}_k\right\rangle P(z)=0,$ whenever, $k\geq q$ and $ z\in C.$
Hence $\left\langle f,\varphi^{n-1}_k\right\rangle=0,$ for all $k\geq q$  if and only if $C\nsubseteq P^{-1}(0).$
Therefore, $f$ is a finite linear combination of $\varphi^{n-1}_k$'s and by the given growth condition on
$f,$ it follows that $f=0.$ Thus, we infer that $f=0$  if and only if $C\nsubseteq P^{-1}(0).$
\end{proof}
\begin{remark}
Theorem \ref{th4} has been work out only for radial class of functions. However, the question that
a complex cone is set of injectivity for the weighted twisted spherical mean for those functions $f$
on $\mathbb C^n$ satisfying the growth condition $ e^{\frac{1}{4}|z|^2}f(z)\in L^p(\mathbb C^n)$
with $1\leq p<\infty$ is still unanswered.
\end{remark}

\smallskip

\bigskip

\noindent{\bf Acknowledgements:}\\
The author wishes to thank E. K. Narayanan
and S. Thangavelu for several fruitful discussions, during my visit to IISc, Bangalore.
The author would also like to gratefully acknowledge the support provided by
IIT Guwagati, Government of India.

\end{document}